\documentclass[preprint,10pt]{my-elsarticle}

\usepackage{amsmath,amssymb,amsthm}
\usepackage{a4wide}
\usepackage{subfig}
\usepackage{bigints}
\newtheorem{proposition}{Proposition}

\newtheorem{theorem}{Theorem}

\journal{\rm This is a preprint published Open Access in 'Applied Numerical Mathematics'
[{\tt https://doi.org/10.1016/j.apnum.2024.09.014}].}

\begin{document}

\begin{frontmatter}

\title{Exact solution for a discrete-time SIR model}

\author[1]{M\'{a}rcia Lemos-Silva}
\ead{marcialemos@ua.pt}
\ead[url]{https://orcid.org/0000-0001-5466-0504}

\author[2]{Sandra Vaz}
\ead{svaz@ubi.pt}
\ead[url]{https://orcid.org/0000-0002-1507-2272}

\author[1,3]{Delfim F. M. Torres\corref{cor1}}
\ead{delfim@ua.pt}
\ead{delfim@unicv.cv}
\ead[url]{https://orcid.org/0000-0001-8641-2505}

\cortext[cor1]{Corresponding author: delfim@ua.pt}


\address[1]{Center for Research and Development in Mathematics and Applications (CIDMA),\newline 
Department of Mathematics,
University of Aveiro,
3810-193 Aveiro, Portugal}

\address[2]{Center of Mathematics and Applications (CMA-UBI), 
Department of Mathematics,\newline
University of Beira Interior,
6201-001 Covilh\~{a}, Portugal}

\address[3]{Research Center in Exact Sciences (CICE), 
Faculty of Sciences and Technology (FCT),\newline
University of Cape Verde (Uni-CV),
7943-010 Praia, Cabo Verde}


\begin{abstract}
We propose a nonstandard finite difference scheme for the 
Susceptible--Infected--Removed (SIR) continuous model. 
We prove that our discretized system is dynamically consistent 
with its continuous counterpart and we derive its exact solution. 
We end with the analysis of the long-term behavior of susceptible, 
infected and removed individuals, illustrating 
our results with examples. In contrast
with the SIR discrete-time model available in the literature,
our new model is simultaneously mathematically and biologically sound.
\end{abstract}


\end{frontmatter}


\section{Introduction}
\label{sec1}

Compartmental models are widely used in mathematical modeling of infectious diseases. 
Modeling such diseases became a topic of great interest when in 1927 
Kermack and McKendrick presented their SIR model \cite{kermack:mckendrick} 
that aimed to study the spread of infectious diseases in a population divided into three groups: 
susceptible, infected, and removed, by tracking the evolution of each group over time. 
Since then, compartmental models have been widely used in the field of epidemiology, biology, 
economy, among others. Nowadays, these models are still of great importance and have undergone 
remarkable advances, with the development of more sophisticated models that better capture 
the complex dynamics of various diseases, see, e.g., 
\cite{LoliPiccolomini2020,MR3409181,vaz:torres1,vaz:torres}.  

In general, the vast majority of these compartmental models are governed by nonlinear 
differential equations for which no exact solutions are known. To overcome this problem, 
there are several numerical methods one can apply to convert such models to their 
discrete counterparts in order to find suitable approximated solutions. However, 
standard difference schemes, such as Euler and Runge--Kutta methods, 
often fail to solve nonlinear systems in a consistent way \cite{consistent}. 
Thus, it is necessary to adopt a scheme that ensures that no numerical instabilities occur. 
Nonstandard finite discrete difference (NSFD) schemes can be used to eliminate 
these instabilities: see, e.g., Mickens' book \cite{mickens:book}. 
This is possible because there are some designed 
laws that systems must satisfy in order to preserve the qualitative properties of the continuous model, 
such as non-negativity, boundedness, stability of the equilibrium points, and others \cite{mickens1}. 
The literature on Mickens-type NSFD schemes is now vast \cite{mickens:applied,vaz:torres1,vaz:torres}.
In particular, we can find many results on NSFD  
to solve biomath models. For example,
in \cite{MR4712579} the authors show that
the NSFD scheme applies directly for mass action-based models 
of biological and chemical processes, while in \cite{MR4692483}
the Mickens' method is used to construct a dynamically consistent 
second-order NSFD scheme for the general 
Rosenzweig--MacArthur predator-prey model,
and in \cite{MR4696804} for a SIR epidemic model.

While for the most realistic models it is impossible 
to find their solutions analytically, in some cases
the exact solution can be obtained 
\cite{scales:solution,gleissner:solution}. 
Recently, the epidemic SIR continuous differential system model
\cite{bailey} was extended to time-dependent coefficients 
and its exact solution computed \cite{scales:solution}.
More interesting, a new discrete-time SIR model was also proposed 
with an exact solution. Unfortunately, the discrete-time model of
\cite{scales:solution} may present negative solutions even when
the initial conditions are positive, which has no biological meaning.
Here we propose a new NSFD discrete-time model that is consistent
with the continuous SIR model of \cite{bailey}. Although we are 
considering one of the simplest models regarding the spread of 
infectious diseases, our model is new, non-classical, and non-standard. 
More precisely, the discrete-time model we propose here corrects 
the recent discrete-time model of \cite{scales:solution}.
Our proposed model is obtained using the non-standard
approach of Mickens \cite{MR4721474}.
Unlike almost all discrete-time models, 
it has an exact solution, and, in contrast with 
\cite{scales:solution}, we are able to prove that all its solutions 
remain non-negative.

This paper is organized as follows. In Section~\ref{sec2}, 
a brief state of the art for continuous and discrete time 
SIR models with exact explicit solutions is presented.  
Our results are then given in Section~\ref{sec3}: 
we propose a new discrete-time autonomous SIR model proving its dynamical consistence. 
In concrete, we prove the non-negativity and boundedness of its solution
(Proposition~\ref{thm:01}), an explicit formula of the exact solution 
(Theorem~\ref{thm:2}), and we investigate the convergence of the model 
under study to the equilibrium points (Theorems~\ref{convergence:zero} 
and \ref{thm:04}). Some illustrative examples that support the obtained 
results are given in Section~\ref{sec4}. We end with Section~\ref{sec5}
of conclusions and some possible directions for future work.


\section{State of the art and our main goal}
\label{sec2}

In \cite{bailey}, Bailey proposed the following Susceptible--Infected--Removed (SIR) model:
\begin{equation}
\label{sir:continuous}
\begin{cases}
x' = -\dfrac{b x y}{x+y}, \\
y' = \dfrac{b x y}{x+y} - cy, \\
z' = cy,
\end{cases}
\end{equation}
where $b, c \in \mathbb{R}^+$, $x(t_0) = x_0 > 0$, $y(t_0) = y_0 > 0$ 
and $z(t_0) = z_0 \geq 0$ for some $t_0 \in \mathbb{R}_0^+$.
Moreover, $x, y, z : \mathbb{R}_{t_0} 
\rightarrow \mathbb{R}^+_0$, where 
$\mathbb{R}_{t_0} := \left\{t \in \mathbb{R} : t \geq t_0\right\}$.
This system is a compartmental model where $x$ 
represents the number of susceptible individuals, 
$y$ the number of infected individuals, and $z$ the number of removed individuals. 
Note that by adding the right-hand side of equations \eqref{sir:continuous}
one has $x'(t) + y'(t) + z'(t) = 0$, that is, $x(t) + y(t) + z(t) = N$ for all $t$,
where the constant $N := x_0 + y_0 + z_0$ denotes the total population under study.
This means that it is enough to solve the first two equations of system 
\eqref{sir:continuous}: after we know $x(t)$ and $y(t)$, we can immediately compute 
$z(t)$ using the equality $z(t) = N - x(t) - y(t)$.

The exact solution of \eqref{sir:continuous} has already been formulated in \cite{gleissner:solution}. 
Assuming $x, y > 0$ and $b, c \in \mathbb{R}^{+}$, the solution is given by 
\begin{equation}
\label{continuous:solution}
\begin{cases}
x(t) = x_0(1 + \kappa)^{\frac{b}{b-c}}(1 + \kappa e^{(b-c)(t-t_0)})^{-\frac{b}{b-c}},\\
y(t) = y_0(1 + \kappa)^{\frac{b}{b-c}}(1 + \kappa e^{(b-c)(t-t_0)})^{-\frac{b}{b-c}}e^{(b-c)(t-t_0)},\\
z(t) = N - (x_0 + y_0)^{\frac{b}{b-c}}(x_0 + y_0e^{(b-c)(t-t_0)})^{-\frac{c}{b-c}},
\end{cases}
\end{equation} 
where $\kappa = \frac{y_0}{x_0}$. The solution \eqref{continuous:solution}
of system \eqref{sir:continuous} is obtained by solving
the first two equations and then using
the constant population property to infer $z$ \cite{gleissner:solution}. 

The authors of \cite{scales:solution} proposed a different method to solve \eqref{sir:continuous}, 
considering the autonomous and non-autonomous model. In the latter case,  
$b, c: \mathbb{R} \rightarrow \mathbb{R}^+$ and the solution can be written as 
\begin{equation}
\label{continuous:solution2}
\begin{cases}
x(t) = x_0\exp\left\{-\kappa\bigintsss_{t_0}^{t} b(s) \left(\kappa 
+  e^{\int_{t_0}^{s} (c-b)(\tau) d\tau} \right)^{-1} ds\right\},\\
y(t) = y_0\exp\left\{\bigintsss_{t_0}^{t} \left(b(s) \left(1 
+ \kappa e^{\int_{t_0}^{s} (c-b)(\tau) d\tau}\right)^{-1} - c(s) \right)ds\right\},\\
z(t) = N - \left(y_0 e^{\int_{t_0}^{t} (b-c)(s)ds} + x_0\right)
\exp\left\{-\kappa\bigintsss_{t_0}^t b(s)\left(\kappa 
+ e^{\int_{t_0}^{s} (c-b)(\tau)d\tau}\right)^{-1}ds\right\},
\end{cases}
\end{equation}
which is identical to \eqref{continuous:solution} if $b,c \in \mathbb{R}^{+}$. 

Moreover, in \cite{scales:solution} the authors 
also present a discrete-time dynamic SIR model given by
\begin{equation}
\label{sir:discrete:dynamic}
\begin{cases}
x(t+1) = x(t) - \frac{b(t)x(t)y(t+1)}{x(t)+y(t)},\\
y(t+1) = y(t) + \frac{b(t)x(t)y(t+1)}{x(t)+y(t)} - c(t)y(t+1),\\
z(t+1) = z(t) + c(t)y(t+1),
\end{cases}
\end{equation}
with $x(t_0) = x_0 > 0$, $y(t_0) = y_0 > 0$ and $z(t_0) = z_0 \geq 0$,
for some $t_0 \in \mathbb{R}_0^+$,
and $t \in \mathbb{Z}_{t_0} := \left\{t_0, t_0+1, t_0+2, \ldots \right\}$, 	
obtaining its exact solution. In this case, the time step is considered to be $h = 1$, 
and then the state at time $t$ is followed by the state at time $t+1$. Here we call attention 
to the fact that system \eqref{sir:discrete:dynamic} fails to guarantee the non-negativity of solutions, which should
be inherent to any epidemiological model, meaning that \eqref{sir:discrete:dynamic} has no biological relevance. 
For example, let $x_0 = 0.6$, $y_0 = 0.4$ and $z_0 = 0$ with 
$b(t) \equiv 1.5$ and $c(t) \equiv 0.1$. Then, $x(1) = -1.2$, which does not make sense.
More generally, the following result holds.

\begin{proposition}
Let $x_0 > 0$, $y_0 > 0$ and $z_0 \geq 0$. If $b(t) > 1 + c(t)$ for all $t$, 
then there exists a $\tau > t_0$ such that the solution 
$x$ of \eqref{sir:discrete:dynamic} is negative, that is, $x(\tau) < 0$.
\end{proposition}

\begin{proof}
From the second equation of \eqref{sir:discrete:dynamic}, we obtain that
$$
y(t+1) = \frac{y(t)\left(x(t)+y(t)\right)}{(1+c(t)) \left(x(t)+y(t)\right) - b(t) x(t)}.
$$
Substituting this expression into the first equation of system \eqref{sir:discrete:dynamic},
direct calculations show that
$$
x(t+1) = x(t)-\frac{b(t)x(t)y(t)}{\left(1+c(t)\right)\left(x(t)+y(t)\right)-b(t)x(t)}.
$$
We conclude that $x(t+1) < 0$ is equivalent to $b(t) > 1 + c(t)$.
\end{proof}

It is the main goal of the present work to provide a new discrete-time model
for \eqref{sir:continuous}, alternative to \eqref{sir:discrete:dynamic},
also with an exact solution but, in contrast with \eqref{sir:discrete:dynamic}, 
with biological meaning, that is, with a non-negative solution for any non-negative
values of the initial conditions.


\section{The discrete-time SIR model}
\label{sec3}

When discretizing a system it is crucial to choose a method that generates a discrete-time system 
whose properties and qualitative behavior is identical to its continuous counterpart. Here we present 
a dynamically consistent Nonstandard Finite Difference (NSFD) scheme for system \eqref{sir:continuous}, 
constructed by following the rules stated by Mickens in \cite{mickens,mickens1}. According to \cite{mickens1}, 
a dynamically consistent NSFD scheme is constructed by following two fundamental rules:
\begin{enumerate}
\item The first order derivative $x'$ is approximated by 
\begin{equation*}
x' \rightarrow \frac{x_{n+1} - \psi(h) x_n}{\phi(h)},
\end{equation*}
where $x_{n}$ and $x_{n+1}$ describe the state of the variable $x$ at times 
$n$ and $n+1$, respectively, and $h$ denotes the time step. Moreover,
$\psi(h)$ and $\phi(h)$ satisfy $\psi(h) = 1 + \mathcal{O}(h)$ 
and $\phi(h) = h + \mathcal{O}(h^2)$, respectively. The same notions are applied 
to the other variables present on system \eqref{sir:continuous}.

\item Both linear and nonlinear terms of the state variables 
and their derivatives may need to be substituted by nonlocal forms. For example,
\begin{equation*}
xy \rightarrow x_{i+1}y_i, \quad x^2 \rightarrow x_{i+1}x_i,
\quad x^2 = 2x^2 - x^2 \rightarrow 2(x_i)^2 - x_{i+1}x_i.
\end{equation*}
\end{enumerate}

In the majority of works on this topic, one can find that 
$\psi(h)$ is chosen to be simply 1. We also make that choice here. 
Throughout our work, the denominator function will be $\phi(h)=h$. 

Our proposal for a NSFD scheme for system \eqref{sir:continuous} is 
\begin{equation}
\label{sir:discrete}
\begin{cases}
\dfrac{x_{n+1} - x_n}{h} = -\dfrac{bx_{n+1}y_n}{x_n + y_n},\\
\dfrac{y_{n+1} - y_n}{h} = \dfrac{bx_{n+1}y_n}{x_n + y_n} - cy_{n+1},\\
\dfrac{z_{n+1} - z_n}{h} = cy_{n+1},
\end{cases}
\end{equation}
where 
\begin{equation}
\label{sir:disc:assump}
b, c \in \mathbb{R}^+, 
\quad x(t_0) = x_0 > 0, 
\quad y(t_0) = y_0 > 0, 
\quad \text{and} \quad 
z(t_0) = z_0 \geq 0. 
\end{equation}
System \eqref{sir:discrete}
can be rewritten, in an equivalent way, as
\begin{equation}
\label{sir:discrete:explicit}
\begin{cases}
x_{n+1} = \dfrac{x_n(x_n + y_n)}{x_n + y_n(1 + bh)},\\ 
y_{n+1} = \dfrac{y_n(1 + bh)(x_n + y_n)}{(1 + ch)(x_n + y_n(1 + bh))},\\
z_{n+1} = \dfrac{ch y_n(1 + bh)(x_n + y_n)}{(1 + ch)(x_n + y_n(1 + bh))} + z_n.
\end{cases}
\end{equation}
In the sequel we define the total population $N$ by 
$$
N := x_{0}+y_{0}+z_{0}.
$$
Observe that $N > 0$.


\subsection{Non-negativity and boundedness of solutions}
\label{subsec1}

One of the most important properties regarding epidemiological models is the need 
to keep all the solutions non-negative over time. Also, in this particular case, 
the total population $N = x + y + z$ must remain constant. Thus, we prove our first result.

\begin{proposition}
\label{thm:01}	
All solutions of \eqref{sir:discrete}--\eqref{sir:disc:assump}
remain non-negative for all $n >0$. Moreover, the discretized scheme  
\eqref{sir:discrete} guarantees that the population remains constant over time:
$x_{n}+y_{n}+z_{n} = N$ for all $n \in \mathbb{N}$.
\end{proposition}

\begin{proof}
Let $(x_0, y_0, z_0)$ be given in agreement with \eqref{sir:disc:assump}.
Since, by definition, all parameters are non-negative, then all equations of system 
\eqref{sir:discrete:explicit}, $x_{n+1}$, $y_{n+1}$, $z_{n+1}$, are clearly non-negative. 
Let $n \in \mathbb{N}_0$ and define $N_n = x_n + y_n + z_n$. Adding the three 
equations of \eqref{sir:discrete}, we have
\begin{equation*}
\frac{N_{n+1} - N_n}{h} = 0 \Leftrightarrow N_{n+1} = N_n.
\end{equation*}
Thus, the population remains constant: $N_n = N$ for any $n \in \mathbb{N}_0$.
\end{proof}


\subsection{Equilibrium points}
\label{subsec2}

To obtain the equilibrium points of system \eqref{sir:discrete} or, 
equivalently, \eqref{sir:discrete:explicit},
one has to find the fixed points $(x^*,y^*,z^*) = F(x^*,y^*,z^*)$
of function 
$$
F(x,y,z) = 
\left(
\dfrac{x(x + y)}{x + y(1 + bh)},
\dfrac{y(1 + bh)(x + y)}{(1 + ch)(x + y(1 + bh))},
\dfrac{ch y(1 + bh)(x + y)}{(1 + ch)(x + y(1 + bh))} + z
\right).
$$
By doing so, it follows that the equilibrium points are 
$(\alpha, 0, N - \alpha)$, with $\alpha \in \mathbb{R}_0^+$, 
being in agreement with the known results of the 
continuous-time model \eqref{sir:continuous} in \cite{scales:solution}.


\subsection{Exact solution}
\label{subsec3}

Here we derive the exact solution of the discrete-time 
dynamical system \eqref{sir:discrete:explicit}. 

\begin{theorem}
\label{thm:2}	
The exact solution of system \eqref{sir:discrete:explicit} is given by
\begin{equation}
\label{solution:discrete}
\begin{cases}
x_n = x_0 \displaystyle\prod_{i=1}^n 
\frac{1 + \bar{\kappa} \xi^{i-1}}{1 + bh + \bar{\kappa}\xi^{i-1}},\\
y_n = \frac{y_0}{\xi^n} \displaystyle\prod_{i=1}^n 
\frac{1 + \bar{\kappa} \xi^{i-1}}{1 + bh + \bar{\kappa} \xi^{i-1}},\\
z_n = N - \left(x_0 + \frac{y_0}{\xi^n}\right)\displaystyle\prod_{i=1}^n 
\frac{1 + \bar{\kappa} \xi^{i-1}}{1 + bh + \bar{\kappa}\xi^{i-1}},
\end{cases}
\end{equation}
where $\bar{\kappa} = \frac{x_0}{y_0}$ and $\xi = \frac{1 + ch}{1 + bh}$.
\end{theorem}

\begin{proof}
We do the proof by induction. For $n = 1$, we have
\begin{equation}
\label{induction_x1}
x_1 = x_0\left(\frac{1 + \overline{\kappa}}{1 + bh + \overline{\kappa}}\right) 
= x_0\left(\frac{x_0 + y_0}{x_0 + y_0(1 + bh)}\right),\,
\text{because $\overline{\kappa} = \frac{x_0}{y_0}$,}
\end{equation}
and 
\begin{equation}
\label{induction_y1}
y_1 = \frac{y_0}{\xi}\left(\frac{1 + \overline{\kappa}}{1 + bh 
+ \overline{\kappa}}\right) = \frac{y_0(1 + bh)}{(1 + ch)} 
\left(\frac{y_0 + x_0}{x_0 + y_0(1+bh)}\right),\, 
\text{by the definition of $\overline{\kappa}$ and $\xi$}. 
\end{equation}
Both \eqref{induction_x1} and \eqref{induction_y1} are true from system
\eqref{sir:discrete:explicit}. Now, let us state the inductive hypothesis 
that \eqref{solution:discrete} holds true for a certain $n = m$. 
We want to prove that it remains valid for $n = m+1$. Thus,  
\begin{equation*}
\begin{aligned}
x_{m+1} &= x_0 \prod_{i=1}^{m+1} 
\frac{1 + \overline{\kappa}\xi^{i-1}}{1 + bh + \overline{\kappa}\xi^{i-1}}\\ 
&= x_0 \prod_{i=1}^{m} \frac{1 + \overline{\kappa}\xi^{i-1}}{1 + bh 
+ \overline{\kappa}\xi^{i-1}}\cdot \frac{1 + \overline{\kappa}\xi^m}{1 + bh 
+ \overline{\kappa}\xi^m} \\ 
&= x_m \left(\frac{1 + \overline{\kappa}\xi^m}{1 + bh 
+ \overline{\kappa}\xi^m}\right),\,\text{by inductive hypothesis}. 
\end{aligned}
\end{equation*}
Note that the first two equations of \eqref{solution:discrete} can be rewritten as 
\begin{equation}
\label{induction_x0}
x_0 = \frac{x_n}{\prod_{i=1}^n \frac{1 +\bar{\kappa} \xi^{i-1}}{1 + bh + \bar{\kappa}\xi^{i-1}}}
\end{equation}
and 
\begin{equation}
\label{induction_y0}
y_0 = \frac{y_n\xi^n}{\prod_{i=1}^n \frac{1 +\bar{\kappa} \xi^{i-1}}{1 + bh + \bar{\kappa}\xi^{i-1}}}. 
\end{equation}
Then, using \eqref{induction_x0} and \eqref{induction_y0}, and the definition 
of $\overline{\kappa}$, we get
\begin{equation}
\label{induction_xm1}
x_{m+1} = x_m\left(\frac{x_m + y_m}{x_m + y_m(1+bh)}\right). 
\end{equation}
Moreover, 
\begin{equation}
\begin{aligned}
y_{m+1} &= \frac{y_0}{\xi^{m+1}}\prod_{i=1}^{m+1} 
\frac{1+\overline{\kappa}\xi^{i-1}}{1+bh+\overline{\kappa}\xi^{i-1}} \\ 
&= \frac{y_0}{\xi^m}\cdot\frac{1}{\xi}\prod_{i=1}^m 
\frac{1+\overline{\kappa}\xi^{i-1}}{1+bh+\overline{\kappa}\xi^{i-1}} 
\cdot \frac{1 +\overline{\kappa}\xi^m}{1+bh+\overline{\kappa}\xi^m} \\ 
&= \frac{y_m}{\xi} \left(\frac{1 +\overline{\kappa}\xi^m}{1+bh
+\overline{\kappa}\xi^m} \right),\,\text{by inductive hypothesis}. 
\end{aligned}
\end{equation}
Likewise, using \eqref{induction_x0} and \eqref{induction_y0} 
and the definition of $\overline{\kappa}$ and $\xi$, we get
\begin{equation}
\label{induction_ym1}
y_{m+1} = \frac{y_m(1+bh)(x_m + y_m)}{(1+ch)(x_m + y_m(1+bh))}.
\end{equation}
To finish, we note that both \eqref{induction_xm1} and \eqref{induction_ym1} are 
also in agreement with system \eqref{sir:discrete:explicit}.
Furthermore, since $N = x_{n} + y_{n} + z_{n}$, that is, 
$z_{n} = N - x_{n} - y_{n}$, we have
\begin{equation*}
z_n = N - \left(x_0 + \frac{y_0}{\xi^n}\right)
\displaystyle\prod_{i=1}^n \frac{1 + \bar{\kappa} 
\xi^{i-1}}{1 + bh + \bar{\kappa} \xi^{i-1}},\, n \geq 1.
\end{equation*}
The proof is complete.
\end{proof}


\subsection{Long-term behavior}
\label{subsec4}

To begin, let us note that $x_n$, as stated in \eqref{solution:discrete}, 
can be rewritten as
\begin{equation}
\label{xn:proof}
x_n = x_0 \displaystyle \prod_{i = 1}^{n}\left( 
\frac{1}{1 + \frac{bh}{1 + \bar{\kappa}\xi^{i-1}}}\right),
\end{equation}
while $y_n$ can be rewritten as
\begin{equation}
\label{yn:proof}
y_n = y_0 \displaystyle \prod_{i=1}^{n} \left(\frac{1}{\xi} 
\cdot \frac{1 + \bar{\kappa} \xi^{i-1}}{1 + bh + \bar{\kappa}\xi^{i-1}} \right)
= y_0 \displaystyle \prod_{i=1}^{n} \left(\frac{1}{\xi 
+ \frac{bh\xi}{1 +\bar{\kappa}\xi^{i-1}}}\right).
\end{equation}

The reproduction number $\mathcal{R}_0$
is one of the most significant threshold
quantities used in epidemiology. It is well-known that 
$$
\mathcal{R}_0 = \frac{b}{c}
$$
for system \eqref{sir:continuous}.
The same happens for all coherent discretizations 
of \eqref{sir:continuous}, in particular to system
\eqref{sir:discrete} or \eqref{sir:discrete:explicit}.
Here we prove that
the extinction equilibrium is asymptotically stable
when $\mathcal{R}_0 \geq 1$ (Theorem~\ref{convergence:zero}); 
and the disease free equilibrium is asymptotically stable
when $\mathcal{R}_0 < 1$ (Theorem~\ref{thm:04}).

\begin{theorem}
\label{convergence:zero}
Let $\mathcal{R}_0 \geq 1$.
Then all solutions \eqref{solution:discrete} 
converge to the equilibrium $(0,0,N)$. 
\end{theorem}  

\begin{proof}
We prove this result in two distinct parts. 

(i) Assume that $\mathcal{R}_0 = 1$, that is,
$b = c$. Then, $\xi = 1$ and from \eqref{xn:proof} we have
\begin{equation*}
x_n = x_0 \displaystyle \prod_{i = 1}^{n} 
\frac{1 + \bar{\kappa}}{1 + bh + \bar{\kappa}} 
= x_0 \left(\frac{1 + \bar{\kappa}}{1 + b h +\bar{\kappa}}\right)^n.
\end{equation*}
Since $\frac{1 + \bar{\kappa}}{1 + bh + \bar{\kappa}} <1$, 
then $\underset{n \to \infty}{\lim}x_{n} =0$. 
The same conclusion is taken regarding $y_{n}$. Moreover, 
since $z_{n} = N - x_{n}- y_{n}$, it is clear that under the condition 
$b = c$ all the solutions \eqref{solution:discrete} 
converge to the equilibrium $(0,0,N)$. 
		
(ii) Let us now consider that $\mathcal{R}_0 > 1$, that is, $b > c$.
It follows that $\xi = \frac{1 + ch}{1 + bh} < 1$. Considering \eqref{xn:proof}, 
we can define $\mathit{a}_{i}$ as
\begin{equation}
\label{a:xn}
\mathit{a}_{i} = \frac{1}{1 + \frac{bh}{1 + \bar{\kappa}\xi^{i-1}}}, 
\quad i \geq 1.
\end{equation}
Thus,
\begin{equation*}
1>\frac{1}{1 + \frac{bh}{1 + \bar{\kappa} \xi^{i-1}}} 
> \frac{1}{1 + \frac{bh}{1 + \bar{\kappa} \xi^i}}
\Rightarrow 1 >\mathit{a}_{i} > \mathit{a}_{i+1}, 
\quad i \geq 1.
\end{equation*}
This means that $0 < x_{n}=x_{0}\mathit{a}_1\mathit{a}_2
\ldots \mathit{a}_n < x_{0}(\mathit{a}_1)^n$. Therefore,  
$\underset{n\to \infty }{\lim} x_n = 0$.
Similarly, consider \eqref{yn:proof} and define $\tilde{a}_{i}$ as
\begin{equation}
\label{a:yn}
\tilde{a}_{i} = \frac{1}{\xi \left(1 + \frac{bh}{1 
+ \bar{\kappa}\xi^{i-1}}\right)}, \quad i \geq 1.
\end{equation}
From \eqref{a:yn} the following relations hold:
\begin{equation*}
\frac{1}{\xi + \frac{bh\xi}{1 + \bar{\kappa} \xi^{i-1}}} 
> \frac{1}{\xi + \frac{bh\xi}{1 + \bar{\kappa} \xi^i}} 
\Rightarrow \tilde{a}_{i} > \tilde{a}_{i+1}, 
\quad i\geq 1.
\end{equation*}
We can also see that $\tilde{a}_{i-1} < \frac{1}{\xi}$ 
for all $i \in \mathbb{N}$. Note that for $\tilde{a}_1$ 
to be smaller than 1, we need to impose strong conditions. In fact,
\begin{equation*}
\begin{aligned}
\tilde{a}_1 < 1 \Rightarrow \frac{1}{\xi\left(1+\frac{bh}{1 + \overline{\kappa}}\right)} < 1 
&\Leftrightarrow \frac{1 + \overline{\kappa}}{1 +\overline{\kappa} + bh} < \frac{1 + ch}{1 + bh},
\quad \text{by definition of $\xi$,} \\ 
&\Leftrightarrow c > \frac{b\overline{\kappa}}{1 + bh + \kappa}. 
\end{aligned} 
\end{equation*}
Nevertheless, we can also see that for a certain $p$, we have $\tilde{a}_{p+1}<1$. Indeed, 
\begin{equation*}
\begin{aligned}
\tilde{a}_p < 1 \Rightarrow \frac{1}{\xi\left(1 + \frac{bh}{1+\overline{\kappa}\xi^{p-1}}\right)} < 1 
&\Leftrightarrow 1 + \overline{\kappa}\xi^{p-1} < \xi(1 + \overline{\kappa}\xi^{p-1} + bh\xi) \\ 
&\Leftrightarrow \xi^{p-1}(\overline{\kappa} - \overline{\kappa}\xi) < \xi(bh + 1) - 1 \\ 
&\Leftrightarrow \xi^{p-1} < \frac{ch}{\overline{\kappa}(1 - \xi)},\quad \text{by definition of $\xi$.}
\end{aligned}
\end{equation*}
Then, applying the natural logarithm to both sides of the inequality, we get
\begin{equation*}
\ln (\xi^{p-1}) < \ln\left(\frac{ch}{\overline{\kappa}(1 - \xi)}\right)
\end{equation*}
and, since $\xi < 1$, it follows that
\begin{equation}
p > 1+ \dfrac{ \ln \left( \frac{c h}{(1-\xi) \bar{\kappa}} \right)}{\ln(\xi)}. 
\end{equation}
Thus, we can rewrite
\begin{equation}
y_{n}=y_{0}\tilde{a}_{1}\tilde{a}_2
\ldots\tilde{a}_n<y_{0}(\tilde{a}_1)^{p}\prod_{i=p+1}^{n} \tilde{a}_i < y_{0} \,M
\left(\tilde{a}_{p+1}\right)^{n-(p+1)}, 
\quad M=(\tilde{a}_1)^{p} \in \mathbb{R}^{+}.
\end{equation}
Since $\underset{n \to \infty}{\lim}\left(\tilde{a}_{p+1}\right)^{n-(p+1)}= 0$, 
this means that $\underset{n\to \infty }{\lim} y_n = 0$. By definition, 
we also have $\underset{n\to \infty }{\lim} z_n = N$. Therefore, 
all the solutions will converge to the equilibrium point $(0,0,N)$ when $b>c$. 

Given (i) and (ii), the proof is complete.
\end{proof}

\begin{theorem}
\label{thm:04}
Consider $\mathcal{R}_0 < 1$.
Then all solutions \eqref{solution:discrete} 
converge to the equilibrium $(\alpha,0,N  - \alpha)$ for some $\alpha \in \, ]0,N]$. 
\end{theorem}

\begin{proof}
Since $\mathcal{R}_0 < 1$, this means that
$b < c$ and $\xi = \frac{1 + ch}{1 + bh} > 1$.
Considering \eqref{xn:proof}, the following relations are valid:
\begin{equation*}
1 + \frac{bh}{1 + \kappa\xi^{i-1}} > 1 + \frac{bh}{1 + \kappa\xi^i}>1, 
\quad i \geq 1,
\end{equation*}
and, defining $\mathit{a}_{i}$ as \eqref{a:xn}, we have
\begin{equation*}
\frac{1}{1 + \frac{bh}{1 + \kappa\xi^{i-1}}} 
< \frac{1}{1 + \frac{bh}{1 + \kappa\xi^i}} 
\Rightarrow \mathit{a}_{i} < \mathit{a}_{i+1}, 
\quad i \geq 1,
\end{equation*}
which means that 
\begin{equation*}
x_{n}=\displaystyle x_{0}\prod_{i=1}^n a_i >x_{0} \prod_{i=1}^n a_1 
\Rightarrow x_{n}=x_{0}\prod_{i=1}^n a_i > x_{0}(a_1)^n.
\end{equation*}
Since $(a_1)^n \underset{n \to \infty}{\longrightarrow} 0$, 
then $x_{n}=x_{0}\prod_{i=1}^n a_i = \alpha$, 
where $\alpha > 0$. Thus, $\displaystyle \lim_{n \rightarrow \infty} x_n = \alpha$ 
with $0 < \alpha \leq N$.
	
Now, let us consider  \eqref{yn:proof} and \eqref{a:yn}. 
Since $\xi > 1$ and $1 + \frac{bh}{1 + \kappa\xi^{i-1}} > 1$, 
we have
\begin{equation*}
\tilde{a}_i = \frac{1}{\xi\left(1 + \frac{bh}{1 
+ \kappa\xi^{i-1}}\right)} < 1. 
\end{equation*}
Moreover,
\begin{equation*}
\frac{1}{\xi + \frac{bh\xi}{1 + \kappa\xi^{i-1}}} 
< \frac{1}{\xi + \frac{bh\xi}{1 + \kappa\xi^i}} <1 
\Rightarrow \tilde{a}_{i} < \tilde{a}_{i+1} < 1, 
\quad i \geq 1. 
\end{equation*}	
Therefore, $0 < y_{n}=y_{0} \prod_{i=1}^n \tilde{a}_i 
<  (\tilde{a}_n)^n$. Since $\tilde{a}_{n}<1$, $(\tilde{a}_{n})^{n}
\underset{n \to \infty}{ \longrightarrow} 0$, so 
$\underset{n \to \infty}{\lim} y_n = 0$. It is now easy 
to conclude that $\underset{n \to \infty}{\lim} z_n 
= N - \alpha$, leading to the conclusion that
all the solutions \eqref{solution:discrete} 
converge to the equilibrium $(\alpha, 0, N-\alpha)$.	
\end{proof}


\section{Illustrative examples}
\label{sec4}

This section is dedicated to some examples concerning 
systems \eqref{sir:continuous} and \eqref{solution:discrete} 
for different parameter values. Note that since we have 
analytical/exact solutions for the considered systems, there is 
no need for numerical methods: all our figures
are obtained with the exact formulas, with no error.
In all Figures~\ref{Fig.1}--\ref{Fig.3} we use $h = 0.05$.

In Figures~\ref{Fig.1} and \ref{Fig.2}, 
we consider the case where $b > c$ ($\mathcal{R}_0 > 1$)
and it can be observed that 
the solution always converges to the equilibrium point $(0,0,N)$.  

\begin{figure}[ht!]
\centering
\subfloat[Solution of the continuous-time model
\eqref{sir:continuous}]{\includegraphics[scale=0.4]{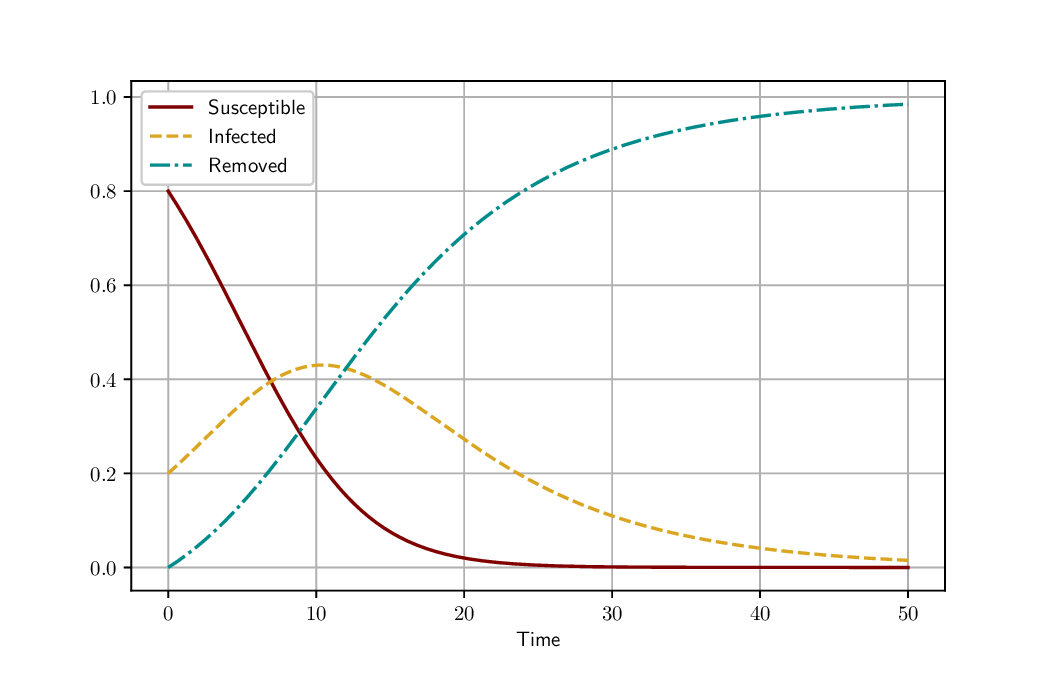}}
\subfloat[Solution \eqref{solution:discrete} of the discrete-time model 
\eqref{sir:discrete} (or \eqref{sir:discrete:explicit}) 
with $h = 0.05$]{\includegraphics[scale=0.4]{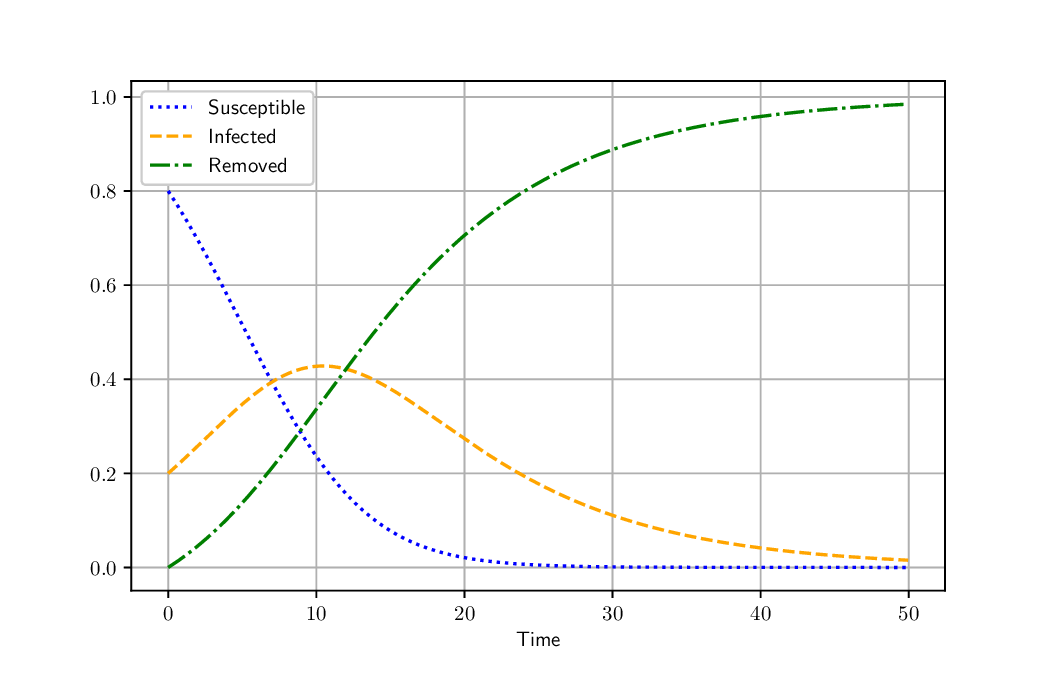}}
\caption{Susceptible, Infected and Removed individuals with 
$b = 0.3$, $c = 0.1$, $x_0 = 0.8$, $y_0 = 0.2$, and $z_0 = 0$.}
\label{Fig.1}
\end{figure}

In Figure \ref{Fig.2}, it is considered the infected for different 
values of $c<b$ and the convergence is always attained.
\begin{figure}[ht!]
\centering
\includegraphics[scale = 0.6]{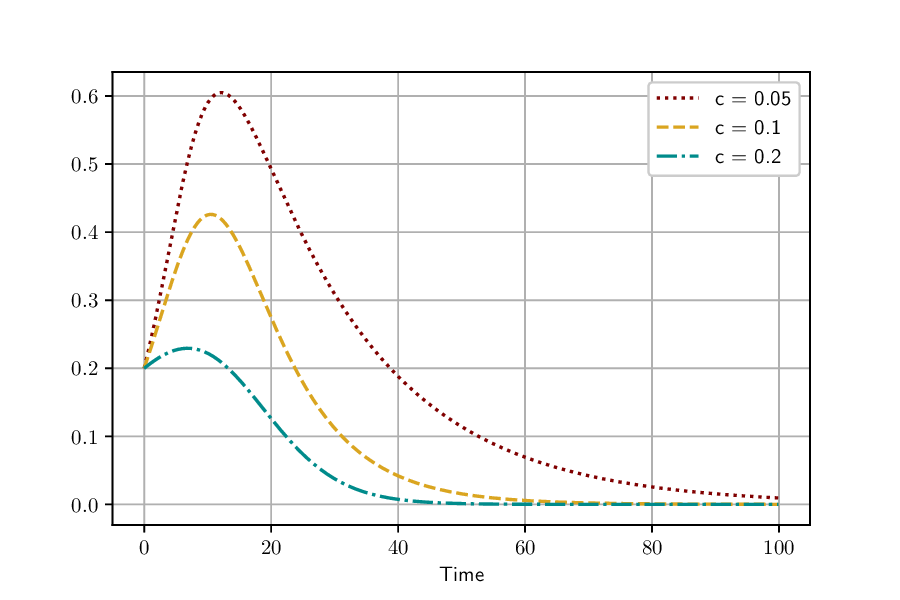} 
\caption{Infected individuals of system \eqref{sir:discrete} (or \eqref{sir:discrete:explicit})
with $b = 0.3$, $c \in \{0.05, 0.1, 0.2\}$, $h = 0.05$,
$x_0 = 0.8$, $y_0 = 0.2$, and $z_0 = 0$.}
\label{Fig.2}
\end{figure}

On the other hand, Figure~\ref{Fig.3} considers the case where 
$b < c$ ($\mathcal{R}_0 < 1$), 
where it is possible to see that the solution converges 
to $(\alpha, 0, N - \alpha)$, where $\alpha \in \mathbb{R}^+$. 
All the results are in agreement with Theorems~\ref{convergence:zero}
and \ref{thm:04}.

\begin{figure}[ht!]
\centering
\subfloat[Solution of the continuous-time model 
\eqref{sir:continuous}]{\includegraphics[scale=0.4]{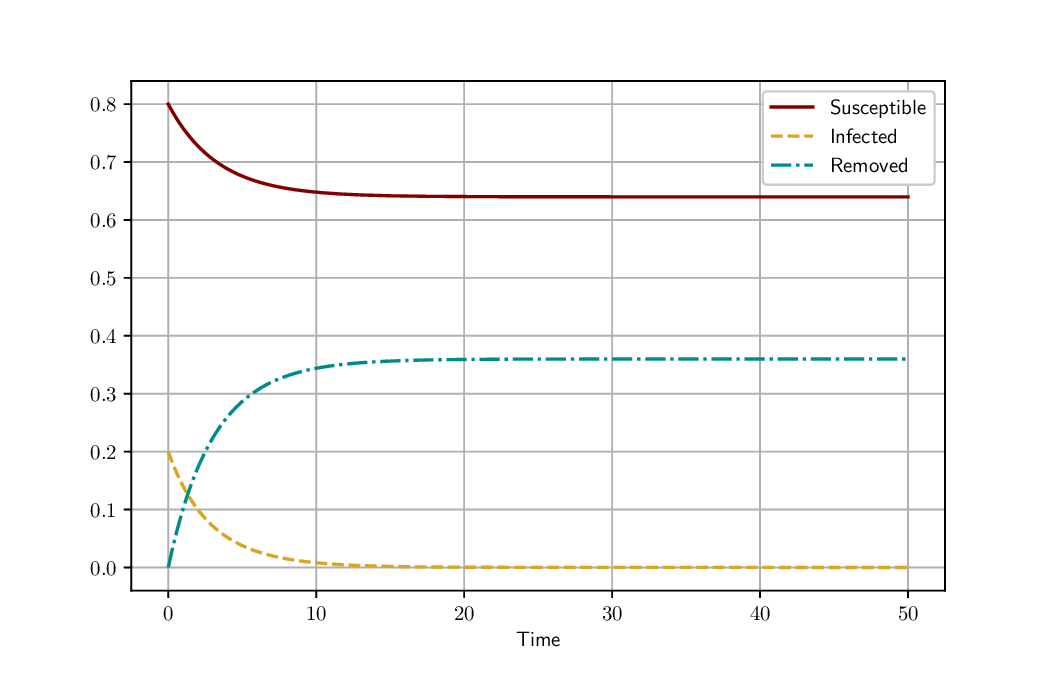}}
\subfloat[Solution \eqref{solution:discrete} 
of the discrete-time model \eqref{sir:discrete} 
(or \eqref{sir:discrete:explicit})
with $h = 0.05$]{\includegraphics[scale=0.4]{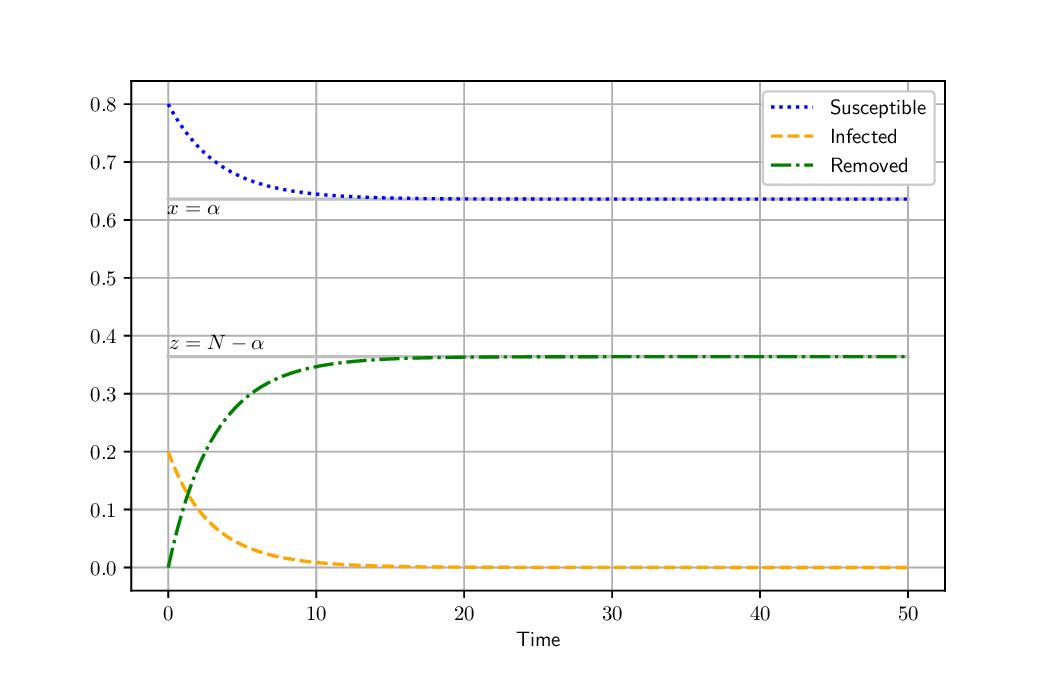}}
\caption{Susceptible, Infected and Removed individuals with 
$b = 0.3$, $c = 0.6$, $x_0 = 0.8$, $y_0 = 0.2$, and $z_0 = 0$ 
($N = 1$, $\alpha \approx 0.636$).}
\label{Fig.3}
\end{figure}

All our solutions were obtained by using $\mathsf{Python}$, 
which is nowadays a popular open choice for scientific computing \cite{MR4539988}.
 
 
\section{Conclusions}
\label{sec5}

By following Mickens' nonstandard finite difference scheme,
in this work we proposed a new discrete-time autonomous SIR model
that, in contrast with the one proposed in \cite{scales:solution},
guarantees the non-negativity of the solutions along time and, thus,
has biological meaning. The equilibrium points were found 
and the non-negativity and boundedness of solutions were proved,
which are in agreement with the standard continuous model.
Moreover, the exact solution of the proposed model was obtained 
and its local stability proved. Finally, some examples
were given that illustrate the obtained theoretical results.

For future work, it would be interesting to obtain the exact solution
of the proposed discrete-time model when the infection and recovery rates 
are not constant. This question is nontrivial and remains open.
It would be also interesting to see a concrete application 
of the proposed model for the analysis of a real biological problem.


\section*{CRediT authorship contribution statement}

All authors listed on the title page have contributed significantly to the work.

\section*{Declaration of competing interest}

The authors have no conflict of interest.

\section*{Data availability}

Not applicable.

\section*{Acknowledgements}

The authors are grateful to the Reviewers
for several constructive remarks, comments and suggestions,
that helped them to improve the quality and clearness of
the submitted manuscript.


\section*{Funding}

The authors were partially supported by 
the Portuguese Foundation for Science and Technology (FCT):
Lemos-Silva and Torres through
the Center for Research and Development in Mathematics 
and Applications (CIDMA), projects UIDB/04106/2020 and UIDP/04106/2020;
Vaz through the Center of Mathematics and Applications 
of Universidade da Beira Interior (CMA-UBI), 
project UIDB/00212/2020. Lemos-Silva was also supported by
the FCT PhD fellowship with reference UI/BD/154853/2023;
Torres within the project 
``Mathematical Modelling of Multiscale Control Systems: 
Applications to Human Diseases'' (CoSysM3), 
Reference 2022.03091.PTDC, financially supported 
by national funds (OE) through FCT/MCTES.



\end{document}